\documentclass[12pt]{article}

\usepackage{amsfonts}
\usepackage{amssymb}
\usepackage{amsmath}
\usepackage{amsthm}
\usepackage{graphicx}
\usepackage{eucal}
\usepackage{ wasysym }
\usepackage{tikz}
\usepackage{tkz-berge} %Allows us to do graphs more easily, e.g., "\Edge(0)(1)"
\usetikzlibrary{arrows,decorations.markings}
\usepackage{hyperref}
\hypersetup{colorlinks=true}
\usepackage{verbatim}
\usepackage{todonotes}
\usepackage{adjustbox}

\theoremstyle{definition}
\newtheorem{lem}{Lemma}[section]

\newtheorem{thm}[lem]{Theorem}
\newtheorem{cor}[lem]{Corollary}

\begin{document}

\title{Some Winnability Results for the Neighborhood and Group Labeling Lights Out Games}

\author{Brittany Doherty\thanks{brittanyrcast@gmail.com}
\and Christian J. Miller\thanks{Department of Mathematics, Michigan State University, East Lansing, MI 48824, USA, mill3706@msu.edu, }
\and Darren B. Parker\thanks{Department of Mathematics, Grand Valley State University, Allendale, Michigan 49401-6495, parkerda@gvsu.edu} }

\date{
%\today \\
\small MR Subject Classifications: 05C20, 06B99\\
\small Keywords: Lights out, light-switching game, winnability, graph game}

\maketitle

\begin{abstract}
We look at both the \emph{group labeling lights out game} and the \emph{neighborhood lights out game}.  Our main focus is to determine necessary and sufficient conditions for when the group labeling lights out game on path graphs, cycle graphs, and complete bipartite graphs can be won for every possible initial labeling.  In the process of solving this problem, we demonstrate a new proof for when the neighborhood lights out game on complete bipartite graphs can be won for every possible initial labeling.
\end{abstract}

\section{Introduction}

The lights out game on graphs or directed graphs is an example of a light-switching game.  In any light-switching game, there is a collection of lights that can be on, off, or perhaps have multiple on-states (which can be interpreted as different colors or different intensities of the same color).  There is also a collection of switches, where each switch can change the states of one or more of the lights when toggled.  The object of the game is usually to get all the lights into the ``off'' state.  Examples of light-switching games are the Berlekamp (or Gale-Berlekamp) light-switching game (see \cite{Brualdi/Meyer:15}, \cite{Carson/Stolarski:04}, and \cite{Schauz:11}) and Merlin's Magic Square (see \cite{Pelletier:87} and \cite{Stock:89}).

Lights out is a commonly studied light-switching game that was originally an electronic game created by Tiger Electronics in 1995.  The idea behind this game has since been extended to several light-switching games on graphs.  Some of these extensions are direct generalizations of the original game, like the $\sigma^+$-game in \cite{Sutner:90}, the neighborhood lights out game developed independently in \cite{Arangala:12} and \cite{paper11}, and a matrix-generated version of the game in \cite{paper15}.  Other versions are explored in \cite{Pelletier:87}, \cite{Araujo:00}, \cite{Craft/Miller/Pritikin:09}) and \cite{paper16}.

Each version of the game begins with some labeling of the vertices, usually by elements of $\mathbb{Z}_m$ for some $m \ge 2$.  The switches in this game are the vertices.  The game is won when we achieve some desired labeling, usually where each vertex has label 0 (i.e. 0 is considered the ``off'' label).  We call this labeling the \emph{zero labeling}.

The most direct generalization of the original lights out game is called the \emph{neighborhood $m$-lights out game}, where $m \ge 2$ is an integer.  In this game, the labels come from $\mathbb{Z}_m$.  When a vertex $v$ is toggled, the labels of $v$ and every vertex adjacent to $v$ are increased by 1 modulo $m$.  As noted above, the game is won when the zero labeling is achieved.  This game is closely linked to the \emph{neighborhood matrix} of $G$, which we denote by $N(G)$, or simply $N$ if it is clear what the graph is.  Thus, we also call the neighborhood $m$-lights out game the $(N,m)$-lights out game.

We also study a version of the lights out game defined in \cite{paper14}, where the vertex labels come from a group $H$.  The game begins with a graph $G$ and an initial \emph{labeling} of the vertices with elements of $H$, which we primarily express as a function $\lambda_0: V(G) \rightarrow H$.  The game is played by toggling vertices.  Each time a vertex is toggled, it changes the labeling of $G$. If at any point in the game we have a labeling $\lambda: V(G) \rightarrow H$, then when a vertex $v$ is toggled, this changes the labeling to $\lambda'$, where for each $w \in V(G)$, we have

\[
\lambda'(w) = \left\{ \begin{matrix} \lambda(v) * \lambda(w), & w = v \hbox{ or } vw \in E(G) \\
\lambda(w), & \hbox{ otherwise.} \end{matrix} \right.
\]
The game is won when we achieve the \emph{identity labeling} where the label of each vertex is the identity element of $H$.  If $H$ is abelian and the binary operation is addition, we also call this the \emph{zero labeling} as in the neighborhood lights out game.  We call this game the \emph{$H$-labeling lights out game} (or, more briefly, the \emph{$H$-labeling game}).

In both of the above games, if we begin the game with an unfortunate labeling, it may be impossible to win the game.  If it is possible to win the game when we begin with the labeling $\lambda_0$, we call $\lambda_0$ an $(N,m)$-winnable labeling in the neighborhood lights out game and an $H$-winnable labeling in the $H$-labeling lights out game.  Depending on the game, we say that $G$ is $(N,m)$-Always Winnable or $H$-Always Winnable (abbreviated $(N,m)$-AW or $H$-AW) if every labeling of $V(G)$ is $(N,m)$-winnable or $H$-winnable, respectively.

For the group labeling lights out game, we focus here on games where the group is a cyclic group $\mathbb{Z}_m$.  Our main goal is to determine which path graphs, cycle graphs, and complete bipartite graphs are $\mathbb{Z}_m$-AW.  In Section~\ref{GeneralResults}, we narrow down the values of $m$ that can possibly support $\mathbb{Z}_m$-AW graphs.  We also narrow down the groups that are necessary to consider when proving a graph is $\mathbb{Z}_m$-AW.  In Section~\ref{Neighborhood}, we show how we can use the neighborhood lights out game to help us find necessary conditions for our graphs to be $\mathbb{Z}_m$-AW.  Along the way, we give a new, simpler proof of $(N,m)$-AW complete bipartite graphs.  Finally, in Section~\ref{PathCycleBipartite}, we prove that the necessary conditions from Section~\ref{Neighborhood} are sufficient as well.

\section{Winnability with the Group $\mathbb{Z}_m$} \label{GeneralResults}

We begin with a result from \cite{paper14} that narrows down the possible cyclic groups that lead to always winnable graphs.

\begin{thm}
\cite[Thm. 2.4]{paper14}
If $m=2^kd$, where $d$ is odd, then there is a one-to-one correspondence between $\mathbb{Z}_m$-winnable labelings of $G$ and $\mathbb{Z}_{2^k}$-winnable labelings of $G$.
\end{thm}

This implies that if $d>1$, then it is impossible for all labelings to be winnable.  We thus get the following.

\begin{cor}
If $G$ is $\mathbb{Z}_m$-AW, then $m = 2^k$ for some $k \ge 1$.
\end{cor}

Thus, for the remainder of the paper, we assume the vertex labels come from $\mathbb{Z}_{2^k}$ for some $k \ge 1$.

For our next results, it is helpful to recall that every element of $\mathbb{Z}_{2^k}$ is a congruence class (i.e. a set of integers that are congruent to one another modulo $2^k$).  Moreover, for each $S \in \mathbb{Z}_{2^k}$, since $2^k$ is even, either every element of $S$ is odd or every element is even.  Thus, it makes sense to say that $S$ is even if every integer in $S$ is even, and $S$ is odd if every integer in $S$ is odd.  For $q \in \mathbb{Q}$, we define $qS = \{ x \in \mathbb{Q} : x = qs$ for some $s \in S \}$.  Note that if $S$ is even, then $\frac12 S \in \mathbb{Z}_{2^{k-1}}$ (see \cite{paper14} for details).  

Thinking of the labels this way helps us view a $\mathbb{Z}_{2^k}$-labeling lights out game as having a simultaneous $\mathbb{Z}_2$-labeling lights out game.  For a graph $G$ and labeling $\lambda: V(G) \rightarrow \mathbb{Z}_{2^k}$, define $\lambda_2: V(G) \rightarrow \mathbb{Z}_2$ by

\[
\lambda_2(v) = \left\{ \begin{matrix} 1, & \lambda(v) \hbox{ is odd} \\
0, & \lambda(v) \hbox{ is even} \end{matrix} \right.
\]

\begin{lem}
\label{2lem}
Let $G$ be a graph, let $\lambda: V(G) \rightarrow \mathbb{Z}_{2^k}$ be a labeling, and let $\lambda_2: V(G) \rightarrow \mathbb{Z}_2$ be as defined above.
\begin{enumerate}
\item \label{setup} If we toggle $v \in V(G)$ under the rules of the $\mathbb{Z}_{2^k}$-labeling game with labeling $\lambda$ to get the labeling $\pi$, then toggling $v$ under the rules of the $\mathbb{Z}_2$-labeling game with labeling $\lambda_2$ results in the labeling $\pi_2$.
\item \label{inductionhelper} Suppose $\lambda_2$ is $\mathbb{Z}_2$-winnable.  If we begin the $\mathbb{Z}_{2^k}$-labeling game with labeling $\lambda$, then it is possible to toggle the vertices of $G$ so that every vertex has an even label.
\end{enumerate}
\end{lem}

\begin{proof}
For (\ref{setup}), we consider the parity of $\lambda(v)$.  If $\lambda(v)$ is even, then when we toggle $v$ in the $\mathbb{Z}_{2^k}$-labeling game, each label of an adjacent vertex is increased by an even number, and labels of vertices not adjacent to $v$ are unchanged.  Thus, the parities of all vertices are unchanged and so $\pi_2=\lambda_2$.  Also, $\lambda_2(v)=0$, and so toggling $v$ in the $\mathbb{Z}_2$-labeling game leaves the labels of all vertices unchanged.  Thus, the resulting labeling is $\lambda_2$, which we determined above to be $\pi_2$.  In the case of $\lambda(v)$ odd, toggling $v$ in the $\mathbb{Z}_{2^k}$-labeling game gives us $\pi$, which is the same as $\lambda$ on all vertices not equal or adjacent to $v$.  The label on $v$ and each vertex adjacent to $v$ changes parity, since we add the odd number $\lambda(v)$ to each of these labels.  In the $\mathbb{Z}_2$-labeling game, $\lambda_2(v)=1$, and so toggling $v$ changes the parity of all adjacent vertices and leaves all other labels unchanged (just like with $\pi$).  We end up with the labeling $\pi_2$.

For (\ref{inductionhelper}), an easy induction using (\ref{setup}) gives us that for any sequence of toggles in the $\mathbb{Z}_{2^k}$-labeling game giving us a labeling $\pi$, the same toggles used in the corresponding $\mathbb{Z}_2$-labeling game gives us the labeling $\pi_2$.  Now suppose we begin the $\mathbb{Z}_{2^k}$-labeling game with labeling $\lambda$.  Since $\lambda_2$ is $\mathbb{Z}_2$-winnable, if we begin with labeling $\lambda_2$, we can toggle the vertices of $G$ in the $\mathbb{Z}_2$-labeling game to achieve the zero labeling.  We now apply these same toggles in the $\mathbb{Z}_{2^k}$-labeling game.  The resulting labeling $\pi$ in the $\mathbb{Z}_{2^k}$-labeling game has $\pi_2$ as the zero labeling, which means that $\pi(v)$ is even for all $v \in V(G)$.  This completes the proof.
\end{proof}

In the case that the range of $\lambda$ consists only of even labels, we get a similar relationship as above between the $\mathbb{Z}_{2^k}$-labeling game and the $\mathbb{Z}_{2^{k-1}}$-labeling game.

\begin{lem}
\label{k-1lem}
Let $G$ be a graph, and let $k \ge 2$.  Suppose $\lambda: V(G) \rightarrow \mathbb{Z}_{2^k}$ is a labeling where each vertex has an even label.  Define $\pi: V(G) \rightarrow \mathbb{Z}_{2^{k-1}}$ by $\pi(w)=\frac12 \lambda(w)$ for all $w \in V(G)$.  If we toggle the vertex $v$ in the $\mathbb{Z}_{2^k}$-labeling game with labeling $\lambda$ to get the labeling $\lambda'$, then toggling $v$ in the $\mathbb{Z}_{2^{k-1}}$-labeling game with labeling $\pi$ results in the labeling $\pi'$, where $\pi'(w) = \frac12 \lambda'(w)$ for all $w \in V(G)$.
\end{lem}

\begin{proof}
When we toggle $v$ in the $\mathbb{Z}_{2^k}$-labeling game with labeling $\lambda$, each adjacent vertex $w$ gets the label $\lambda'(w)=\lambda(v)+\lambda(w)$.  Toggling $v$ in the $\mathbb{Z}_{2^{k-1}}$-labeling game with labeling $\pi$ gives us $\pi'(w)=\pi(v)+\pi(w) = \frac12 \lambda(v) + \frac12 \lambda(w)$.  It is easy to show that (as sets) $\frac12 \lambda(v) + \frac12 \lambda(w) = \frac12 ( \lambda(v)+\lambda(w)) = \frac12 \lambda'(w)$.  Thus, $\pi'(w) = \frac12 \lambda'(w)$.  Since vertices not adjacent to $v$ are unchanged in both games, we get $\pi'(w) = \frac12 \lambda'(w)$ for them as well.
\end{proof}

We use Lemma~\ref{2lem} and Lemma~\ref{k-1lem} to prove a result that allows us to assume our vertex labelings come from $\mathbb{Z}_2$.

\begin{thm}
\label{reduce2}
Let $G$ be a graph and let $k \ge 1$.  Then $G$ is $\mathbb{Z}_{2^k}$-AW if and only if $G$ is $\mathbb{Z}_2$-AW.
\end{thm}

\begin{proof}
First we assume $G$ is $\mathbb{Z}_{2^k}$-AW.  To prove $G$ is $\mathbb{Z}_2$-AW, we let $\lambda: V(G) \rightarrow \mathbb{Z}_2$ be a labeling and prove $\lambda$ is $\mathbb{Z}_2$-winnable.  We define $\pi: V(G) \rightarrow \mathbb{Z}_{2^k}$ by $\pi(v) = \lambda(v)$ mod $2^k$.  It is clear from the definition that $\pi_2=\lambda$.  Since $G$ is $\mathbb{Z}_{2^k}$-AW, $\pi$ is $\mathbb{Z}_{2^k}$-winnable.  Thus, we can toggle the vertices in the $\mathbb{Z}_{2^k}$-labeling game to get the zero labeling, which we will call $\pi'$.  By Lemma~\ref{2lem}(\ref{setup}), if we do this same toggling in the $\mathbb{Z}_2$-labeling game with initial labeling $\lambda = \pi_2$, we end up with the labeling $\pi'_2$.  Since $\pi'$ is the zero labeling, all labels are even, and so $\pi'_2$ is the zero labeling as well.  This proves that $\lambda$ is $\mathbb{Z}_2$-winnable, and so $G$ is $\mathbb{Z}_2$-AW.

For the other direction, assume $G$ is $\mathbb{Z}_2$-AW and let $\lambda: V(G) \rightarrow \mathbb{Z}_{2^k}$.  We prove $\lambda$ is $\mathbb{Z}_{2^k}$-winnable by induction.  The case $k=1$ follows from the assumption that $G$ is $\mathbb{Z}_2$-AW.  For $k>1$, by Lemma~\ref{2lem}(\ref{inductionhelper}), since $G$ is $\mathbb{Z}_2$-AW, we can toggle the vertices in the $\mathbb{Z}_{2^k}$-lights out game so that all vertices have even labels.  Thus, we can assume $\lambda(v)$ is even for all $v \in V(G)$.  Since all labels of $\lambda$ are even, we can define $\pi: V(G) \rightarrow \mathbb{Z}_{2^{k-1}}$ by $\pi(w)=\frac12 \lambda(w)$.  

By the induction hypothesis, $\pi$ is $\mathbb{Z}_{2^{k-1}}$-winnable, so we can toggle the vertices in the $\mathbb{Z}_{2^{k-1}}$-labeling game so that we obtain the zero labeling $\pi_0$.  If we toggle the vertices identically in the $\mathbb{Z}_{2^k}$-labeling game, an easy induction using Lemma~\ref{k-1lem} implies that this results in a labeling $\lambda'$ such that $\pi_0(w) = \frac12 \lambda'(w)$ for all $w \in V(G)$.  Since $\pi_0(w)=0$ for all $w \in V(G)$, it follows that $\lambda'(w)=0$ as well, making $\lambda'$ the zero labeling.  Thus, $\lambda$ is $\mathbb{Z}_{2^k}$-winnable.  Since $\lambda$ is arbitrary, that makes $G$ $\mathbb{Z}_{2^k}$-AW, proving the theorem. 
\end{proof}

\section{Neighborhood Lights Out Game and $\mathbb{Z}_{2^k}$-Winnability} \label{Neighborhood}

Our main results for the $\mathbb{Z}_{2^k}$-labeling game will be characterizations of $\mathbb{Z}_{2^k}$-AW path graphs, cycle graphs, and complete bipartite graphs.  This requires us to prove our conditions on the graphs are both necessary and sufficient for the graphs to be $\mathbb{Z}_{2^k}$-AW.  We have already simplified our task a great deal.  Theorem~\ref{reduce2} implies that we need only find necessary and sufficient conditions for our graphs to be $\mathbb{Z}_2$-AW.  In this section, we use the neighborhood lights out game to help us find necessary conditions for winnability.

The neighborhood lights out game is advantageous to our study of the group labeling game for two reasons.  First, the $(N,2)$-lights out game plays very similarly to the $\mathbb{Z}_2$-labeling game.  In the $\mathbb{Z}_2$-labeling game, when we toggle a vertex whose label is 1, that changes its own label to 0 and the label of each adjacent vertex from 0 to 1 or from 1 to 0.  This is identical to the $(N,2)$-lights out game.  Also, if a vertex has label 0, then toggling it in the $\mathbb{Z}_2$-labeling game has no effect on the labeling of the graph.  Thus, the $\mathbb{Z}_2$-labeling game is equivalent to playing the $(N,2)$-lights out game where we only toggle vertices that have label 1.  This means that if we want even a chance to win the $\mathbb{Z}_2$-labeling game, we have to be able to win the $(N,2)$-lights out game as well.  This gives us the following.

\begin{lem} \label{winnecessary}
If a graph $G$ is $\mathbb{Z}_2$-AW, then $G$ is $(N,2)$-AW.
\end{lem}

The second advantage of the $(N,2)$-lights out game is that winnability is generally easier to determine than in the $\mathbb{Z}_2$-labeling game.  We illustrate this on a general result for complete bipartite graphs.  We should note that $(N,m)$-AW complete bipartite graphs were characterized in \cite{paper11}.  However, here we give a much more elegant proof.

Recall that a complete bipartite graph is a graph whose vertices can be partitioned into two sets, $P_1$ and $P_2$, such that $E(G) = \{ vw : v \in P_1$ and $w \in P_2 \}$.  Before proving our main result, we must prove the following.

\begin{lem} \label{nbipartitesetup}
Let $k,n,p \ge 1$ and $m \ge 2$. 
\begin{enumerate}
\item \label{bipartitestandardform} Suppose we are playing either the $(N,m)$-lights out game or the $\mathbb{Z}_{2^k}$-labeling game on $K_{n,p}$.  If $P_1$ and $P_2$ are the parts from the partition of $V(K_{n,p})$ (with $|P_1|=n$ and $|P_2|=p$), then for any initial labeling of $V(G)$, we can toggle the vertices in such a way that all vertices in $P_1$ have label 0 and all vertices in $P_2$ have the same label.
\item \label{reducevariables} Suppose that an $(N,m)$-winnable labeling has all vertices in $P_1$ (respectively $P_2$) having the same label.  In a winning toggling, each vertex in $P_1$ (resp. $P_2$) is toggled the same number of times.
\end{enumerate}
\end{lem}

\begin{proof}
For (\ref{bipartitestandardform}), in both games we obtain the desired labeling by first toggling each vertex in $P_2$ until it has label 0.  This can be done in the neighborhood game because each toggle increases the label of the toggled vertex by 1 modulo $m$, so we get to label 0 within $k-1$ toggles.  This can be done in the $\mathbb{Z}_{2^k}$-labeling game since each toggle doubles the label of the toggled vertex, so we get to label 0 within $k$ toggles.  Since no two vertices in $P_2$ are adjacent, this results in all vertices in $P_2$ having label 0.  Then we toggle each vertex in $P_1$ until it has label 0.  As before, this results in each vertex in $P_1$ having label 0.  Also, since every vertex in $P_1$ is adjacent to every vertex in $P_2$, every vertex in $P_2$ will have the same label, since each toggle of a vertex in $P_1$ increases every vertex in $P_2$ by the same number.

For (\ref{reducevariables}), assume that all vertices in $P_1$ have the same label.  For contradiction, assume two vertices $v,w \in P_1$ are toggled a different number of times for a winning toggling.  Once this is done, $v$ and $w$ have different labels.  This will still hold when all vertices in $P_1$ are toggled, since none of them are adjacent to $v$ or $w$.  Each time we toggle a vertex in $P_2$, this increases the labels of $v$ and $w$ by 1 modulo $m$.  Thus, once all vertices in $P_2$ are toggled, $v$ and $w$ still have different labels.  But now all vertices have been toggled, and it is impossible for both $v$ and $w$ to have label 0.  This contradicts our assumption of having a winning toggling, and so $v$ and $w$ must be toggled the same number of times.  The proof for all vertices of $P_2$ having the same label is similar.
\end{proof}

We are now ready to characterize $(N,m)$-AW complete bipartite graphs.

\begin{thm}
\label{nbipartite}
The graph $K_{n,p}$ is $(N,m)$-AW if and only if $\gcd(m,np-1)=1$.
\end{thm}

\begin{proof}
We begin by assuming $K_{n,p}$ (with parts $P_1$ and $P_2$) is $(N,m)$-AW and proving that $\gcd(m,np-1)=1$.  Let $\lambda: V(G) \rightarrow \mathbb{Z}_m$ be any labeling.  By Lemma~\ref{nbipartitesetup}(\ref{bipartitestandardform}), we can assume that there is some $a \in \mathbb{Z}_m$ such that $\lambda(v)=0$ for all $v \in P_1$ and $\lambda(v)=a$ for all $v \in P_2$.  By Lemma~\ref{nbipartitesetup}(\ref{reducevariables}) and the fact that $K_{n,p}$ is $(N,m)$-winnable, there exist $x,y \in \mathbb{Z}$ such that toggling every vertex in $P_1$ $x$ times and toggling every vertex in $P_2$ $y$ times results in the zero labeling.  Considering each vertex individually, each vertex in $P_1$ gets its label increased by each vertex in $P_2$ by $y$.  Each such vertex is also increased by its own toggling, which increases its label by $x$  Since there are $p$ vertices in $P_2$, and since each vertex in $P_1$ begins with label 0, the terminal label for each vertex in $P_1$ is $x+py$.  Similarly, the terminal label for each vertex in $P_2$ is $nx+y+a$.  Note that the $a$ appears in this last expression because the initial label of each vertex of $P_2$ is $a$.

Putting together the expressions we get for the terminal labels of the vertices, along with the fact that each terminal label is 0, we get the equations $x+py=0$ and $nx+y+a=0$.  Since $\lambda$ is $(N,m)$-winnable, this system of equations must always have a solution.  Conversely, if a solution of this system of equations exists, then if each vertex of $P_1$ is toggled $x$ times and each vertex of $P_2$ is toggled $y$ times, this is a winning toggling for the game.  Thus, $K_{n,p}$ is $(N,m)$-AW if and only if the system $x+py=0$ and $nx+y+a=0$ has a solution.  By solving the system directly or computing the determinant of $\left[ \begin{matrix} 1 & p \\ n & 1 \end{matrix} \right]$, a solution always exists if and only if $np-1$ is a unit in $\mathbb{Z}_m$.  This occurs precisely when $\gcd(m,np-1) =1$.
\end{proof}

If we let $m=2^k$, the condition $\gcd(2^k,np-1)=1$ is equivalent to one or both of $n$ and $p$ being even.  Putting this together with Lemma~\ref{winnecessary} and Lemma~\ref{reduce2}, we get the following.

\begin{cor}
\label{bipartitenecessary}
If $K_{m,n}$ is $\mathbb{Z}_{2^k}$-AW, then one or both of $m$ and $n$ is even.
\end{cor}

In \cite{paper11}, winnability in the neighborhood lights out game was also determined for paths and cycles for all $(N,m)$-lights out games.  If we state these results in the case that $m=2$, we get that $P_n$ is $(N,2)$-AW precisely when $n \equiv 0$ or 1 (mod 3) and $C_n$ is $(N,2)$-AW precisely when $n \equiv 1$ or 2 (mod 3).  Putting these results together with Lemma~\ref{winnecessary} gives us the following.

\begin{lem}
\label{pathcyclenecessary}
Let $n,k \in \mathbb{N}$.
\begin{enumerate}
\item If $n \ge 1$ and $P_n$ is $\mathbb{Z}_{2^k}$-AW, then $n \equiv 0$ or 1 (mod 3).
\item If $n \ge 3$ and $C_n$ is $\mathbb{Z}_{2^k}$-AW, then $n \equiv 1$ or 2 (mod 3).
\end{enumerate}
\end{lem}

\section{Sufficient Conditions for $\mathbb{Z}_{2^k}$-Winnability} \label{PathCycleBipartite}

Corollary~\ref{bipartitenecessary} and Lemma~\ref{pathcyclenecessary} give necessary conditions for $P_n$, $C_n$, and $K_{n,p}$ to be $\mathbb{Z}_{2^k}$-AW.   In this section, we show that these conditions are sufficient as well.  The notation we use for $P_n$ and $C_n$ is $V(P_n) = V(C_n) = \{ v_1, v_2, \ldots , v_n \}$, $E(P_n) = \{ v_iv_{i+1} : 1 \le i \le n-1 \}$, and $E(C_n) = E(P_n) \cup \{ v_1v_n \}$.

We prove a given condition is sufficient for a graph to be $\mathbb{Z}_2$-AW by proving each possible labeling is $\mathbb{Z}_2$-winnable.  The first step in doing this is to prove we can convert each labeling into a labeling that is easy to work with, much like Lemma~\ref{nbipartitesetup}(\ref{bipartitestandardform}).  We do this in the following lemma.

\begin{lem} \label{sufficientsetup}
Let $n \in \mathbb{N}$, and suppose we are playing the $\mathbb{Z}_2$-labeling game on $P_n$ or $C_n$.
\begin{enumerate}
\item \label{pathsetup} For any initial labeling of $P_n$, we can toggle the vertices to achieve a labeling $\lambda$ with $\lambda(v_i)=0$ for all $i \ge 2$.
\item \label{cyclesetup} If $n \ge 3$, then for any initial labeling of $C_n$, we can toggle the vertices to achieve a labeling $\lambda$ with $\lambda(v_i)=0$ for all $i \ge 3$.
\end{enumerate}
\end{lem}

\begin{proof}
For (\ref{pathsetup}), let $\pi$ be a labeling on $P_n$.  We prove our result by induction on the maximum number $k$ such that $\pi(v_k) = 1$ (i.e. the ``last'' vertex to have a nonzero label).  If $k=1$, then $\pi(v_i)=0$ for all $i \ge 2$, and so $\pi$ is the desired labeling.  For $k>1$, we have two cases.  If $k=n$, then we toggle $v_k=v_n$.  This causes $v_k$ to have label 0, which means the maximum $i$ with $v_i$ having a nonzero label is at most $k-1$.  By induction, we can toggle the vertices to achieve labeling $\lambda$ with $\lambda(v_i)=0$ for all $j \ge 2$.  In the case $k<n$, we consider the subcases $\pi(v_{k-1})=1$ and $\pi(v_{k-1})=0$.  If $\pi(v_{k-1})=1$, we toggle $v_{k-1}$.  This changes the labels of $v_{k-1}$ and $v_k$ to 0, which makes the maximum $i$ with $v_i$ having a nonzero label at most $k-2$.  We then apply induction as in the $k=n$ case.  In the case $\pi(v_{k-1}) =0$, we first toggle $v_k$.  This changes the labels of $v_k$ to 0 and of $v_{k-1}$ and $v_{k+1}$ to 1.  We then toggle $v_{k-1}$.  This changes the labels of $v_{k-1}$ to 0 and $v_k$ to 1.  Finally, we toggle $v_k$ again, which changes the labels of $v_k$ and $v_{k+1}$ to 0, and $v_{k-1}$ to 1.  This gives us a labeling where the maximum $i$ with $v_i$ having a nonzero label is $k-1$.  We then apply induction as in the previous cases, and the lemma is proved.  The proof for (\ref{cyclesetup}) is almost identical.  The only change is that the base cases are $k=0$ and $k=1$.
\end{proof}

Now we look to prove the conditions sufficient for the $\mathbb{Z}_2$-labeling game.  We begin with path graphs.

\begin{lem}
\label{pathcyclesufficient}
Let $n \in \mathbb{N}$.
\begin{enumerate}
\item \label{pathsufficient} If $n \ge 1$ and $n \equiv 0$ or 1 (mod 3), then $P_n$ is $\mathbb{Z}_2$-AW.
\item \label{cyclesufficient} If $n \ge 3$ and $n \equiv 1$ or 2 (mod 3), then $C_n$ is $\mathbb{Z}_2$-AW.
\end{enumerate}
\end{lem}

\begin{proof}
For (\ref{pathsufficient}), we assume $\lambda$ is a labeling of $P_n$ and prove $\lambda$ is $\mathbb{Z}_2$-winnable.  By Lemma~\ref{sufficientsetup}(\ref{pathsetup}), we can assume $\lambda(v_i)=0$ for $i \ge 2$.  If $\lambda(v_1)=0$, we have the zero labeling, making $\lambda$ $\mathbb{Z}_2$-winnable, so we can assume $\lambda(v_1)=1$.  We begin our toggling strategy by toggling each $v_i$ in order from $v_1$ to $v_n$.  An easy induction implies that for all $1 \le i \le n-1$, when we toggle $v_i$, every $v_j$ with $j<1$ has label 1, $v_{i+1}$ has label 1, and all other vertices have label 0.  When we toggle $v_n$ all vertices have label 1, except $v_n$, which has label 0.  

Now we consider the cases $n \equiv 0$ (mod 3) and $n \equiv 1$ (mod 3) separately.  Suppose $n \equiv 0$ (mod 3).  For each $0 \le s \le r-1$, we then toggle $v_{3s+1}$.  For $s=0$, this changes the labels of $v_1$ and $v_2$ to 0.  For each $1 \le s \le r-1$, this changes the labels of $v_{3s}$, $v_{3s+1}$, and $v_{3s+2}$ to 0.  This results in the zero labeling, which makes $\lambda$ $\mathbb{Z}_2$-winnable.

If $n \equiv 1$ (mod 3), then we toggle each $v_{3s+2}$ for all $0 \le s \le r-1$.  This changes the labels of each $v_{3s+1}$, $v_{3s+2}$, and $v_{3s+3}$ to 0, which again results in the zero labeling.  Thus, $\lambda$ is $\mathbb{Z}_2$-winnable, and so $P_n$ is $\mathbb{Z}_2$-AW.

The proof for (\ref{cyclesufficient}) is similar.  We assume $\lambda$ is a labeling of $C_n$ and prove $\lambda$ is $\mathbb{Z}_2$-winnable.  By Lemma~\ref{sufficientsetup}(\ref{cyclesetup}), we can assume $\lambda(v_i)=0$ for all $i \ge 3$.  If $\lambda(v_1)=\lambda(v_2)=0$, then $\lambda$ is the zero labeling and is thus $\mathbb{Z}_2$-winnable.  If $\lambda(v_1)=\lambda(v_2)=1$, then we can toggle $v_2$ to get $\lambda(v_i)=0$ for all $i \ne 3$.  By relabeling the vertices starting at $v_3$, we can thus assume $\lambda(v_1)=1$ and $\lambda(v_i)=0$ for all $i \ne 1$.

In the case $n \equiv 1$ (mod 3), we have $n=3r+1$ for some $r \ge 1$.  Similarly as the proof for $P_n$, we toggle each $v_i$ in order from $v_1$ to $v_{n-2}=v_{3r-1}$.  As in the proof for $P_n$, for each $i<n-2$ and $i=n-1$, $v_i$ has label 1.  Since $v_n$ is adjacent to $v_1$, then $v_n$ also has label 1.  Thus, $v_i$ has label 1 for all $i \ne n-2$ and $v_{n-2}$ has label 0.  We toggle each $v_{3s}$ for $1 \le s \le r$.  Using similar reasoning as in the proof above for $P_n$, $n \equiv 1$ (mod 3), this results in the zero labeling.

In the case $n \equiv 2$ (mod 3), we have $n=3r+2$ for some $r \ge 1$.  We toggle each $v_i$ in order from $v_1$ to $v_{n-1}=v_{3r+1}$.  Using similar reasoning as before, every vertex has label 1 except $v_{n-1}$ and $v_n$, which both have label 0.  By toggling all $v_{3s+2}$ for $0 \le s \le r-1$, we get the zero labeling.  Thus, $\lambda$ is $\mathbb{Z}_2$-winnable, and so $C_n$ is $\mathbb{Z}_2$-AW.
\end{proof}

Our next result is for complete bipartite graphs.

\begin{lem}
\label{bipartitesufficient}
If $n,p \in \mathbb{N}$ and one or both of $n$ and $p$ is even, then $K_{n,p}$ is $\mathbb{Z}_2$-AW.
\end{lem}

\begin{proof}
Let $P_1$ and $P_2$ be the parts of $K_{n,p}$ with $|P_1|=n$ and $|P_2|=p$.  We assume $\lambda$ is a labeling of $K_{n,p}$ and prove that $\lambda$ is $\mathbb{Z}_2$-winnable.  By Lemma~\ref{nbipartitesetup}(\ref{bipartitestandardform}), we can assume that all vertices in $P_1$ have label 0 and all vertices in $P_2$ have the same label.  If the vertices of $P_2$ have label 0, that makes $\lambda$ the zero labeling, which is obviously $\mathbb{Z}_2$-winnable, so we can assume each vertex in $P_2$ has label 1.  Without loss of generality, we can assume $p$ is even (if not, we can toggle each vertex in $P_2$ to make every vertex in $P_2$ have label 0 and every vertex in $P_1$ have label 1).

We then achieve the zero labeling by toggling each vertex in $P_2$ once.  Each toggling increases the labels of each vertex in $P_1$ by 1 mod 2.  Since there is an even number of vertices in $P_2$, each vertex in $P_1$ has its label increased by 0 mod 2.  Thus, each vertex in $P_1$ ends up with label 0.  Since each vertex in $P_2$ changes its own label to 0, this results in the zero labeling.  Thus, $\lambda$ is $\mathbb{Z}_2$-winnable, which makes $K_{n,p}$ $\mathbb{Z}_2$-AW.
\end{proof}

We can now state our main result, which follows directly from Theorem~\ref{reduce2}, Lemma~\ref{pathcyclenecessary}, Lemma~\ref{pathcyclesufficient}, and Lemma~\ref{bipartitesufficient}.

\begin{thm}
\label{pathcycleequivalent}
Let $k,n,p \in \mathbb{N}$.
\begin{enumerate}
\item \label{pathequivalent} If $n \ge 1$, then $P_n$ is $\mathbb{Z}_{2^k}$-AW if and only if $n \equiv 0$ or 1 (mod 3).
\item \label{cycleequivalent} If $n \ge 3$, then $C_n$ is $\mathbb{Z}_{2^k}$-AW if and only if $n \equiv 1$ or 2 (mod 3).
\item \label{bipartiteequivalent} If $n,p \ge 1$, then $K_{n,p}$ is $\mathbb{Z}_{2^k}$-AW if and only if one or both of $n$ and $p$ is even.
\end{enumerate}
\end{thm}

\section{Open Problems}

We close with three possible directions for further research.
\begin{itemize}
\item When we set out to determine whether or not our graphs were $\mathbb{Z}_{2^k}$-AW, we used a two-step process.  We first looked at the $(N,2)$-lights out game to find necessary conditions for a graph to be $\mathbb{Z}_{2^k}$-AW.  Then we proved that these conditions are sufficient as well.  However, this would not work if we had a graph whose winnability in the $\mathbb{Z}_2$-labeling game was different that in the $(N,2)$-lights out game.  Are there any graphs that are $(N,2)$-AW but not $\mathbb{Z}_2$-AW?
\item If the answer to the above question is yes, that complicates our determination of necessary conditions for a graph to be $\mathbb{Z}_{2^k}$-AW.  What makes studying the neighborhood lights out game so nice is that winning or losing the game depends only on how many times we toggle each vertex, not the order in which we toggle the vertices.  Along with some other nice properties of the game, this makes it possible to determine winnability using systems of linear equations.  In the $H$-labeling lights out game, this is not the case (see \cite{paper14} for a detailed discussion).  If we cannot depend on the neighborhood lights out game to help us prove group labeling lights out theorems, what other techniques can we use?
\item While there seems to be much to learn from the group labeling games with cyclic groups, it would be nice to see how the game works with non-cyclic groups.  By the same reasoning that led us to restricting our attention to $\mathbb{Z}_{2^k}$, we can only have $H$-winnable graphs when $|H|=2^k$ for some $k \in \mathbb{N}$.  But there are no other obvious restrictions.  Non-abelian groups are especially tempting to look at.
\end{itemize}

\bibliographystyle{amsalpha}
\bibliography{lightsout}

\end{document}